\documentclass{amsart}
\usepackage{amsmath,amssymb,amsthm}
\usepackage{amsaddr}
\usepackage[backend=biber,style=alphabetic]{biblatex}
\newtheorem{theorem}{Theorem}[section]
\newtheorem{lemma}[theorem]{Lemma}
\theoremstyle{definition}
\newtheorem{definition}[theorem]{Definition}
\newtheorem{prop}[theorem]{Proposition}
\theoremstyle{remark}
\newtheorem{remark}[theorem]{Remark}

\begin{document}

\title{A note on isomorphism types of complete discretely valued fields of mixed characteristic}
\author{Pranav Velavan}
\email{pranavelavan@outlook.com}

\begin{abstract}
In this paper, we resolve the problem of finding the required level of truncation needed on the isomorphism class of a truncated valuation ring of a complete discretely valued field of mixed characteristic to determine the isomorphism class of the field itself.
\end{abstract}

\maketitle
\pagestyle{plain}

\section{Introduction}

In the study of complete discretely valued fields, it is a natural question to investigate to what extent a field is determined by a sufficient truncation of its valuation ring. Significant work on this problem has been completed in the case where the fields of study have perfect residue fields (see for instance \cite{LEE2021102927}) where the required algebraic machinery is simpler. In the narrower case of finite extensions of $\mathbb{Q}_p$, similar questions have been explored: see \cite{Hou-Keating}. The general case which includes that of imperfect residue fields too is more subtle and less work has been done. The general technique to study these fields is to realise them as totally ramified extensions over some unramified subfield. Then we can identify features of the field by studying properties of the resulting Eisenstein polynomial. In recent work by Anscombe-Dittmann-Jahnke \cite{Dittmann}, it was established that the level of truncation needed to determine a complete discretely valued field depends only on the residue field characteristic $p$ and initial ramification $e$. Further, they achieved a bound on this quantity $n(e,p)$ by proving a uniform bound on the discriminant on all Eisenstein polynomials that could occur. In this paper, we determine $n(e,p)$ exactly in all cases.\\

We are primarily interested in the following quantity.
\begin{definition}
     $n(e,p)$ is the smallest positive integer $n$ satisfying:\[
     (K,v)\cong (L,w)\iff\mathcal{O}_v/(p)^n\cong\mathcal{O}_w/(p)^n
     \] for all pairs of complete discretely valued fields $(K,v)$ and $(L,w)$ of mixed characteristic $(0,p)$ and initial ramification $e$.

\end{definition}
Note:
\begin{lemma}
    $n(e,p)$ exists for all $e,p$. Furthermore\[
    n(e,p) \leq e(1+v_p(e))
    \]
\end{lemma}

\begin{proof}
    See \cite{Dittmann} Corollary 5.5
\end{proof}
    
We prove the following theorem:

\begin{theorem}
    Suppose $e>1$ and let $k= v_p(e)$. Then:
    \[
     n(e,p)= k+2
    \] for $p\neq 2$ or $k=0$ and 
    \[
    n(e,2)=k+3
    \] otherwise.
    
\end{theorem}
\begin{remark}
    In the case $e=1$, it is known already that an unramified complete discretely valued field of mixed characteristic is determined up to isomorphism by its residue field; i.e in the language of this paper $n(1,p) =1$. See \cite{serre1979local} Chapter II §5 Theorem 3 for the perfect residue field case and \cite{CML_2022__14_2_1_0} Corollary 6.6 for the general case.
\end{remark}
\begin{remark}
    This is an improvement on Theorem 4.4 in \cite{LEE2021102927} for perfect residue fields. Note that \cite{LEE2021102927} studies powers of the maximal ideal $\mathfrak{m}$ whilst we consider powers of the ideal $(p)$ so the lifting number $M(R)$ from \cite{LEE2021102927} differs by a factor of the ramification index $e$.
\end{remark}
\section{Preliminaries for finitely ramified complete discretely valued fields}
The general reference for the results in this section is \cite{serre1979local}.
We fix a prime number $p$. All valued fields ($K,v$) will be of mixed characteristic $(0,p)$. The value group and residue field of $K$ will be denoted $vK$, $Kv$ respectively. The algebraic closure of a field $K$ will be denoted $K^{al}$. Hereafter $v_p$ will denote the renormalisation of $v$ satisfying $v_p(p)=1$.

\begin{definition}[Initial Ramification]
The \textit{initial ramification} of a complete discretely valued field ($K,v$) of mixed characteristic $(0,p)$ is defined to be $e:=v(p).$
    
\end{definition}

We note the following important lemma:

\begin{lemma}\label{cohenlem}
    Let $(K,v)$ be a complete discretely valued field of initial ramification $e$. Then there exists an unramified subfield $K_0$ such that $Kv = K_0v$ and $K/K_0$ is a totally ramified field extension of degree $e$.
\end{lemma}
\begin{proof}
By \cite{cohen1946structure}, Theorem 11, the valuation ring $\mathcal O_v$ contains a Cohen subring $C_K$ whose residue field is naturally identified with $Kv$. Put $K_0:=\operatorname{Frac}(C_K)$. Since $C_K$ is a strict Cohen ring, it is a complete discrete valuation ring with uniformiser $p$. Let $\pi$ be a uniformiser of $K$. Since $v(p)=e$, we have $p\mathcal O_v=\mathfrak m_v^e=\pi^e\mathcal O_v$. As $C_K/pC_K\cong Kv$, it follows that $1,\pi,\ldots,\pi^{e-1}$ form a basis of $\mathcal O_v/p\mathcal O_v$ over $C_K/pC_K$. We claim that they in fact form a $C_K$-basis of $\mathcal O_v$. Let $x\in\mathcal O_v$. Choose $c_{i,0}\in C_K$, for $0\leq i<e$, such that $x-\sum_{i=0}^{e-1}c_{i,0}\pi^i\in p\mathcal O_v$. Thus $x=\sum_{i=0}^{e-1}c_{i,0}\pi^i+px_1$ for some $x_1\in\mathcal O_v$. Repeating this construction, for every $N\geq 1$ we obtain elements $c_{i,r}\in C_K$ such that
$x=\sum_{i=0}^{e-1}\left(\sum_{r=0}^{N-1}p^r c_{i,r}\right)\pi^i+p^N x_N$. Since $C_K$ is $p$-adically complete, the sums $a_i:=\sum_{r=0}^{\infty}p^r c_{i,r}$ converge in $C_K$, and hence $x=\sum_{i=0}^{e-1}a_i\pi^i$. Thus $1,\pi,\ldots,\pi^{e-1}$ generate $\mathcal O_v$ over $C_K$. They are also linearly independent. Indeed, suppose that $\sum_{i=0}^{e-1}a_i\pi^i=0$ with $a_i\in C_K$. For each nonzero $a_i$, write $a_i=p^{r_i}u_i$ with $u_i\in C_K^\times$. Since a unit of $C_K$ is also a unit of $\mathcal O_v$, we have $v(a_i\pi^i)=er_i+i$. These values are pairwise distinct for distinct $i$, since they are distinct modulo $e$. Hence a nonzero sum cannot vanish, by the ultrametric inequality. Therefore all $a_i$ are zero. Consequently $\mathcal O_v$ is a free $C_K$-module of rank $e$, and taking fraction fields gives $[K:K_0]=e$. Moreover, $K$ and $K_0$ have the same residue field, while $(vK:vK_0)=e$. Hence $K/K_0$ is totally ramified of degree $e$.
\end{proof}
Next we note the following well-known fact about totally ramified extensions:
\begin{lemma}\label{lem}
    Let $(K,v)$ be a complete discretely valued field and let $(L,w)/(K,v)$ be a totally ramified extension of degree $e$; i.e $[L:K] = (wL:vK) = e$. Then any uniformiser of $L$ generates $L$ over $K$ and its minimal polynomial over $K$ is an Eisenstein polynomial of degree $e$.
\end{lemma}
\begin{proof}
    See \cite{serre1979local} Chapter 1, §6, Proposition 18
\end{proof}

We will also need the following lemmas:
\begin{lemma}[Krasner's Lemma]
    Let $(K,v)$ as before. Let $\alpha,\beta\in K^{al}$. If $v(\alpha-\beta)>v(\alpha-\alpha_i)$ for all conjugates $\alpha_i$ of $\alpha$ over $K$, then $K[\alpha]\subset K[\beta]$.
\end{lemma}
\begin{proof}
    See \cite{milneANT} §7, Prop 7.60.
\end{proof}
\begin{lemma}\label{compatibility} Let $(K,v)$ and $(L,w)$ be fields as above, both with initial ramification $e$, and let $n$ be a positive integer. Suppose $\phi:\mathcal O_v/(p)^n\to\mathcal O_w/(p)^n$ is an isomorphism. Then for any Cohen subring $C_K\subseteq\mathcal O_v$ as in Lemma \ref{cohenlem}, there exists a Cohen subring $C_L\subseteq\mathcal O_w$ and an isomorphism $C_K\to C_L$ whose reduction modulo $(p)^n$ agrees with the restriction of $\phi$. \end{lemma}
\begin{proof}
Let $q_K:\mathcal O_v\to \mathcal O_v/p^n\mathcal O_v$ and $q_L:\mathcal O_w\to \mathcal O_w/p^n\mathcal O_w$ be the quotient maps, and let $\bar\phi:Kv\xrightarrow{\sim}Lw$ be the residue-field isomorphism induced by $\phi$.

Choose a Cohen subring $C_K\subseteq\mathcal O_v$. Since units of $C_K$ remain units in $\mathcal O_v$, we have $C_K\cap p^n\mathcal O_v=p^nC_K$. Hence $\overline C_K:=q_K(C_K)\cong C_K/p^nC_K$ is a Cohen ring of characteristic $p^n$ with residue field $Kv$.

Set $D:=\phi(\overline C_K)\subseteq\mathcal O_w/p^n\mathcal O_w$. Then $D$ is a Cohen ring of characteristic $p^n$ whose residue field is $Lw$. We first show that $D$ is the reduction of a Cohen subring of $\mathcal O_w$.

Define $R:=q_L^{-1}(D)\subseteq\mathcal O_w$. The ring $R$ is local, with maximal ideal $\mathfrak m_R=q_L^{-1}(pD)$. Moreover, $\mathfrak m_R=pR+p^n\mathcal O_w$. Indeed, if $x\in\mathfrak m_R$, then $q_L(x)=pd$ for some $d\in D$. Choosing $r\in R$ with $q_L(r)=d$, we obtain $x-pr\in p^n\mathcal O_w$.

We claim that $R$ is a complete Noetherian local ring. Since the initial ramification of $L$ is $e$, we have $p^n\mathcal O_w=\mathfrak m_w^{en}$. The filtration of $\mathcal O_w/p^n\mathcal O_w$ by the powers of $\mathfrak m_w$ has $en$ successive quotients, each isomorphic to $Lw=D/pD$. Hence $\mathcal O_w/p^n\mathcal O_w$ has finite length as a $D$-module.

It follows that $\mathcal O_w$ is finitely generated as an
$R$-module. Indeed, choose a finite set of $D$-module generators of $\mathcal O_w/p^n\mathcal O_w$ containing $1$, and lift them to $\mathcal O_w$. Every element of $\mathcal O_w$ differs from
an $R$-linear combination of these lifts by an element of $p^n\mathcal O_w\subseteq R$, and hence these lifts generate $\mathcal O_w$ as an $R$-module.

Since $\mathcal O_w$ is Noetherian and is finite as an $R$-module, the Eakin--Nagata theorem implies that $R$ is Noetherian. Furthermore, $R$ is closed in $\mathcal O_w$, and $\mathfrak m_w^{en}\subseteq\mathfrak m_R\subseteq\mathfrak m_w$, so the $\mathfrak m_R$-adic and $\mathfrak m_w$-adic topologies on $R$ are equivalent. Hence $R$ is complete. Thus $R$ is a complete Noetherian local ring.

By Cohen's structure theorem (\cite{cohen1946structure}, Theorem 11), $R$ contains a coefficient ring $C_L$. Since $R\subseteq\mathcal O_w$ has characteristic zero, $C_L$ is a strict Cohen ring. Let $\overline C_L:=q_L(C_L)$. Again, $C_L\cap p^n\mathcal O_w=p^nC_L$, so $\overline C_L\cong C_L/p^nC_L$ is a Cohen ring of characteristic $p^n$. Since $C_L\subseteq R$, we have $\overline C_L\subseteq D$.

In fact $\overline C_L=D$. To see this, choose a $p$-basis $B$ of $Lw$ and choose representatives of $B$ in $\overline C_L$. These are also representatives of $B$ in $D$. Applying Mac Lane's Identity Theorem \cite[Theorem 5.1]{CML_2022__14_2_1_0}, with ambient Cohen ring $D$, gives $D\subseteq\overline C_L$. Hence $D=\overline C_L$.

We may therefore regard the restriction of $\phi$ as an isomorphism of Cohen rings $\phi_C:C_K/p^nC_K\xrightarrow{\sim}C_L/p^nC_L$.

Choose a $p$-basis $B_K$ of $Kv$ and representatives $s_K:B_K\to C_K$. For each $b\in B_K$, choose a representative $s_L(\bar\phi(b))\in C_L$ whose image modulo $p^n$ is $\phi_C\bigl(s_K(b)\bmod p^n\bigr)$. By \cite[Corollary 6.5]{CML_2022__14_2_1_0}, there is a unique isomorphism of Cohen rings $\psi:C_K\xrightarrow{\sim}C_L$ which induces $\bar\phi$ on the residue fields and satisfies $\psi(s_K(b))=s_L(\bar\phi(b))$ for every $b\in B_K$.

Reducing modulo $p^n$ gives an isomorphism $\psi_n:C_K/p^nC_K\xrightarrow{\sim}C_L/p^nC_L$. The maps $\psi_n$ and $\phi_C$ induce the same residue-field isomorphism and agree on the prescribed representatives of the $p$-basis. Applying the uniqueness statement of \cite[Corollary 6.5]{CML_2022__14_2_1_0} to these Cohen rings of characteristic $p^n$, we obtain $\psi_n=\phi_C$.

Thus the reduction of $\psi$ modulo $p^n$ agrees with the restriction of $\phi$, as required.

\end{proof}

Using these we can prove:
\begin{prop}[Lifting]
    Let $(K,v)$, $(L,w)$ be a pair of finitely ramified complete discretely valued fields of initial ramification $e$, let $K_0$ be as in Lemma \ref{cohenlem}, and let $\pi$ be a uniformiser of $K$. Define
    $\kappa(\pi)=\max_{\substack{\pi'\sim\pi\\ \pi'\neq\pi}}v_p(\pi-\pi')$,
    where $\pi'$ runs over all conjugates of $\pi$ over $K_0$, and let
    $S=\sum_{\pi'\neq\pi}v_p(\pi-\pi')$. Suppose $n>\kappa(\pi)+S$ and that $\mathcal{O}_v/(p)^n\cong\mathcal{O}_w/(p)^n$. Then $(K,v)\cong(L,w)$.
\end{prop}

\begin{remark}
    The analogous lifting results in the case of perfect residue fields can be found in \cite{LEE2021102927} Chapter 3.
\end{remark}

\begin{proof}
    Let $F$ denote the minimal polynomial of $\pi$ over $K_0$. Taking
    $K_0:=\operatorname{Frac}(C_K)$ and
    $L_0:=\operatorname{Frac}(C_L)$, and using Lemma
    \ref{compatibility}, we may identify $K_0$ with $L_0$ in such a
    way that the truncation isomorphism is compatible with this
    identification. Let $\bar\pi$ denote the class of $\pi$ in
    $\mathcal O_K/(p^n)$, and choose $x\in\mathcal O_L$ lifting its
    image in $\mathcal O_L/(p^n)$. Then $v_p(F(x))\geq n$.
    
    Let $\pi=\pi_1,\ldots,\pi_e$ be the conjugates of $\pi$ over $K_0$,
    and choose $j$ such that
    $m:=v_p(x-\pi_j)=\max_i v_p(x-\pi_i)$. For $i\neq j$, the
    ultrametric inequality gives
    $v_p(x-\pi_i)\leq v_p(\pi_i-\pi_j)$. Since $K_0$ is complete, its
    valuation has a unique extension to an algebraic closure, and hence
    $\sum_{i\neq j}v_p(\pi_j-\pi_i)=S$ and $\max_{i\neq j}v_p(\pi_j-\pi_i)=\kappa(\pi)$. Therefore
    $v_p(F(x))=\sum_i v_p(x-\pi_i)\leq m+S$.
    
    Since $v_p(F(x))\geq n>\kappa(\pi)+S$, we have
    $m>\kappa(\pi)$. Krasner's lemma now gives
    $K_0(\pi_j)\subseteq K_0(x)\subseteq L$. Since both
    $K_0(\pi_j)$ and $L$ have degree $e$ over $K_0$, equality holds.
    Thus $L=K_0(\pi_j)\cong K$.
\end{proof}

\section{Tamely Ramified Case}
We first consider the case of tame ramification. We define the following standard terminology:
\begin{definition}\label{wilddefn}
    We say a finitely ramified complete discretely valued field $(K,v)$ of mixed characteristic $(0,p)$ and initial ramification $e$ is \textit{tamely ramified} if $p\nmid e$. Otherwise we say $(K,v)$ is \textit{wildly ramified}.
\end{definition}
In this case, the Eisenstein polynomial from Lemma \ref{lem} takes a simple shape:
\begin{lemma}\label{tame_minpoly}
    Let $(K,v)$ be a tamely ramified complete discretely valued field with initial ramification $e$. Then we can choose a uniformiser $\pi$ (from \ref{lem}) with minimal polynomial over $K_0$ (defined as in Lemma \ref{cohenlem}) of the form: \[
    F = X^e -pa
  \] where $a\in \mathcal{O}_0^\times$ is a unit in the Cohen subring $\mathcal{O}_0$.
\end{lemma}

\begin{proof}
    See \cite{lang1995N} Chapter II §5, Proposition 12 (see also \cite{Dittmann} Remark 4.8(2)).
\end{proof}
\begin{remark}
    We note that the hypotheses in Chapter 2 of Lang assume the residue field is perfect. However the proof does not use this hypothesis and hence is valid for our purposes.
\end{remark}
We will need the following formulation of Hensel's Lemma:
\begin{lemma}[Hensel's lemma]\label{Hensel}
Let $(K,v)$ be a complete discretely valued field, and let $f\in\mathcal O_K[X]$. Suppose that there exists $x\in\mathcal O_K$ such that
\[
v(f(x))>2v(f'(x)).
\]
Then there exists $\alpha\in\mathcal O_K$ such that
\[
f(\alpha)=0.
\]
\end{lemma}
\begin{proof}
    See \cite{CF} Chapter II, Appendix C.
\end{proof}
\begin{lemma}
    $F$ (as in Lemma \ref{tame_minpoly}) has a root in a complete discretely valued field $(L,w)\iff F$ has a root in $\mathcal{O}_w/p^2$.
\end{lemma}
\begin{proof}
Recall that $F(X)=X^e-pa$, where $a$ is a unit and $e$ is coprime to $p$. The forward implication is immediate. Conversely, suppose that $\overline{x}\in \mathcal{O}_w/p^2$ is a root of $F$, and choose a lift $x\in\mathcal{O}_w$. Then $F(x)=x^e-pa\in p^2\mathcal{O}_w$, so $w(F(x))\geq 2w(p)$. Since $w(pa)=w(p)$ and $w(x^e-pa)>w(pa)$, the ultrametric inequality gives $ew(x)=w(x^e)=w(p)$. Hence $w(x)=w(p)/e$. As $(e,p)=1$, we have $w(e)=0$, and therefore $w(F'(x))=w(ex^{e-1})=(e-1)w(x)=((e-1)/e)w(p)$. Thus $w(F(x))\geq 2w(p)>2((e-1)/e)w(p)=2w(F'(x))$. By Hensel's lemma (Lemma \ref{Hensel}), $F$ has a root in $\mathcal{O}_w$, and hence in $L$. 
\end{proof}
Hence we obtain:
\begin{lemma}
    For tamely ramified complete discretely valued fields:\[n(e,p) \leq2\]
\end{lemma}
\begin{proof}
     Let $(K,v)$ and $(L,w)$ be tamely ramified complete discretely valued fields of initial ramification $e$, and suppose that $\mathcal{O}_K/p^2\cong \mathcal{O}_L/p^2$. By Lemma \ref{tame_minpoly}, after identifying the corresponding Cohen subrings, we may choose a uniformiser $\pi$ of $K$ whose minimal polynomial over the Cohen subfield $K_0$ is $F(X)=X^e-pa$, with $a\in\mathcal{O}_0^\times$. The image of $\pi$ under the isomorphism $\mathcal{O}_K/p^2\cong \mathcal{O}_L/p^2$ is a root of $F$ in $\mathcal{O}_L/p^2$. By the preceding lemma, $F$ therefore has a root $\pi'\in L$. Since $F$ is Eisenstein of degree $e$, we have $[K:K_0]=[K_0(\pi'):K_0]=e$. Hence the embedding $K=K_0(\pi)\to L$ sending $\pi$ to $\pi'$ has image of ramification index $e$. Since $L$ also has initial ramification $e$, it follows that this embedding is an isomorphism. Thus $\mathcal{O}_K/p^2\cong\mathcal{O}_L/p^2$ implies $K\cong L$, and hence $n(e,p)\leq 2$. 
\end{proof}
Now we need to prove:
\begin{lemma}
    Suppose $e$ and $p$ are coprime and $e>1$. Then:\[
    n(e,p) > 1
    \]
\end{lemma}
\begin{proof}
    It suffices to construct for each prime $p$ and integer $e$ coprime to $p$ a pair of non-isomorphic complete discretely valued fields with isomorphic first truncation. To this end let $f$ be the smallest integer satisfying $e\mid p^f-1$. Let $K$ be the unramified extension of degree $f$ over $\mathbb{Q}_p$. Then $K$ has residue field $k = \mathbb{F}_{p^f}$. Now $e\mid p^f-1 = |k^\times|$ so $k^\times$ contains elements that are not $e$th powers (as $k^\times$ is cyclic). Let $\bar{u}$ be such an element and let $u$ be a lift of $\bar{u}$ in $\mathcal{O}_K^\times$.\\
    Now let $K_1 = K(\alpha)$ where $\alpha^e = p$ and $K_2 = K(\beta)$ where $\beta^e = up$. Since $u$ is a unit, $v_p(p) = v_p(up) = 1$ and hence the minimal polynomials of $\alpha,\beta$ are Eisenstein and the two extensions are degree $e$ tamely totally ramified. We have that $\mathcal{O}_{K_1}\cong\mathcal{O}_K[X]/(X^e-p)$ and $\mathcal{O}_{K_2}\cong\mathcal{O}_K[X]/(X^e-up)$. Reducing modulo $p$, $\mathcal{O}_{K_1}/p\mathcal{O}_{K_1}\cong\mathcal{O}_K[X]/(X^e-p,p)\cong\frac{\mathcal{O}_K/p\mathcal{O}_K[X]}{(X^e)}\cong k[X]/(X^e).$ Also $\mathcal{O}_{K_2}/p\mathcal{O}_{K_2}\cong\mathcal{O}_K[X]/(X^e-up,p)\cong\frac{\mathcal{O}_K/p\mathcal{O}_K[X]}{(X^e)}\cong k[X]/(X^e).$ And therefore $\mathcal{O}_{K_1}/(p)\cong\mathcal{O}_{K_2}/(p)$.\\ Now suppose $K_1\cong K_2$. $K_2$ must contain an element $\gamma$ satisfying $\gamma^e = p$. Let $c = \frac{\beta}{\gamma}$. Then $c^e = \frac{\beta^e}{\gamma^e}=\frac{up}{p}=u$. Reducing modulo the maximal ideal of $\mathcal O_{K_2}$,
    we obtain $\bar c^e=\bar u$ in $k$, contradicting the choice of $\bar u$.\\ Hence $K_1\ncong K_2$.
\end{proof}
Combining the previous two lemmas we obtain:
\begin{theorem}\label{tame}
    With hypotheses as above:\[
    n(e,p) = 2
    \]
\end{theorem}

\section{Wildly Ramified Case}

In this section we consider the case of wild ramification (see Definition \ref{wilddefn}). Let $(K,v)$ have initial ramification $e=np^k$, where $(n,p)=1$. The main proposition proved in this section is:
\begin{prop}\label{kras_bound}
    Let $L/K$ be a degree $e$ totally ramified extension of complete discretely valued fields of mixed characteristic $(0,p)$ and let $\pi$ be a uniformiser of $L$. Then: \[
    \kappa(\pi)\leq \frac{1}{p-1} +\frac{1}{e}
    \]
\end{prop}
To prove this, we will need the help of the following two results:
\begin{lemma}\label{perfect}
    Let $K$ be a complete discretely valued field with imperfect residue field. Then there exists an extension of fields $K^g/K$ such that $K^g$ is a complete discretely valued field with perfect residue field and $e(K^g/K) =1.$
\end{lemma}
\begin{proof} 
Let $k$ be the residue field of $K$. By \cite{cohen1946structure} Lemma 12, p.~77, there exists a complete local ring $R$ containing $\mathcal{O}_K$, unramified with respect to $\mathcal{O}_K$, whose residue field is the perfect closure $k^{\mathrm{perf}}$ of $k$. Let $\pi$ be a uniformiser of $K$. Since $R$ is unramified with respect to $\mathcal{O}_K$, we have $\mathfrak{m}_R=(\pi)$. Since $R$ is a complete Noetherian local ring, the Krull intersection theorem gives $\bigcap_{n\geq 0}\mathfrak{m}_R^n=0$. Hence, for every nonzero $x\in R$, there exists $n\geq 0$ such that $x\in(\pi^n)\setminus(\pi^{n+1})$. Thus $x=\pi^n u$ for some $u\in R\setminus(\pi)=R^\times$. Since $\pi$ is nonzero, being an element of $\mathcal O_K\subseteq R$, it follows that the product of any two nonzero elements of $R$ is nonzero. Hence $R$ is an integral domain. Thus $R$ is a complete discrete valuation ring with uniformiser $\pi$. Setting $K^g=\operatorname{Frac}(R)$, the residue field of $K^g$ is $k^{\mathrm{perf}}$, and, since $\pi$ is a uniformiser of both $K$ and $K^g$, we have $e(K^g/K)=1$.
\end{proof}

\begin{lemma}\label{pkras}
    Let $L/K$ be a degree $e$ totally ramified extension of complete discretely valued fields of mixed characteristic $(0,p)$ both having perfect residue fields and let $\pi$ be a uniformiser of $L$. Then: \[
    \kappa(\pi)\leq \frac{1}{p-1}+\frac{1}{e}
    \]
\end{lemma}
\begin{proof}
    The statement in the case where $K=\mathbb{Q}_p$ can be found in \cite{Hou-Keating} Lemma 2.2. The proof given there holds in its same form for the general case of $K$ being a discretely valued field of mixed characteristic $(0,p)$ with perfect residue field. This is because the machinery used in the proof relies on the hypotheses and results of Chapter IV of \cite{serre1979local}. In this chapter, the fields considered are finite extensions of complete discretely valued fields such that the residue field extension is separable. This condition is clearly satisfied in our case as every  finite extension of a perfect field is separable.
\end{proof}
\begin{proof}[Proof of Proposition \ref{kras_bound}]
    If $K$ has perfect residue field, the result follows directly from Lemma \ref{pkras}. Suppose therefore that the residue field of $K$ is imperfect, and let $K^g$ be as in Lemma \ref{perfect}. Since $e(K^g/K)=1$, we may normalise the valuation $v^g$ on $K^g$ so that $v^g|_K=v$. Let $F\in K[X]$ be the minimal polynomial of $\pi$ over $K$. By Lemma \ref{lem}, $F$ is Eisenstein of degree $e$. Hence $F$ remains Eisenstein over $K^g$. Put $L^g=K^g(\pi)$. Then $F$ is the minimal polynomial of $\pi$ over $K^g$, so $L^g/K^g$ is totally ramified of degree $\deg(F)$. In particular, since $K^g$ has perfect residue field, so does $L^g$.
    
    Fix an embedding $\overline{K}\hookrightarrow\overline{K^g}$ over $K$, and let $\overline{v^g}$ be an extension of $v^g$ to $\overline{K^g}$. Then $\overline{v^g}|_{\overline{K}}$ is an extension of $v$ to $\overline{K}$. Since $K$ is complete, it is Henselian, and $v$ has a unique extension to $\overline{K}$. Thus $\overline{v^g}|_{\overline{K}}=\overline{v}$. Moreover, since $F$ remains irreducible over $K^g$, the conjugates of $\pi$ over $K^g$ are precisely the roots of $F$, hence are the same as its conjugates over $K$. Therefore all valuations of the differences $\pi-\pi_i$, and in particular the quantities occurring in the definition of $\kappa(\pi)$, are unchanged after passing from $K$ to $K^g$. The required bound now follows by applying the perfect residue field case to $L^g/K^g$.
\end{proof}

Putting it all together we have:
\begin{theorem}\label{mainbound}
Suppose $k=v_p(e)\geq1$. Then $n(e,p)\leq k+2+\frac{1}{p-1}$.
\end{theorem}

\begin{proof}
Let $(K,v)$ be an arbitrary complete discretely valued field of initial ramification $e$, let $K_0$ be as in Lemma \ref{cohenlem}, and choose a uniformiser $\pi$ of $K$. By Lemma \ref{lem}, the extension $K/K_0$ is totally ramified of degree $e$ and $\pi$ generates $K$ over $K_0$.

By \cite[Chapter III, \S6, Corollary 2 to Proposition 11]{serre1979local}, $\sum_{\pi'\neq\pi}v(\pi-\pi')=d_{K/K_0}$, where $\pi'$ runs over the conjugates of $\pi$ over $K_0$ and $d_{K/K_0}$ denotes the different exponent. By \cite[Chapter III, \S6, Remark 1 following Proposition 13]{serre1979local}, $d_{K/K_0}\leq e-1+v(e)$. Since $v_p=v/e$, we obtain $S=\sum_{\pi'\neq\pi}v_p(\pi-\pi')\leq \frac{e-1+v(e)}{e}=1-\frac{1}{e}+v_p(e)=k+1-\frac{1}{e}$.

On the other hand, Proposition \ref{kras_bound} gives $\kappa(\pi)\leq\frac{1}{p-1}+\frac{1}{e}$. Hence $\kappa(\pi)+S\leq k+1+\frac{1}{p-1}$.

Therefore, by the Lifting Proposition, every integer $N$ satisfying $N>k+1+\frac{1}{p-1}$ has the required lifting property, uniformly for all complete discretely valued fields of initial ramification $e$. Taking $N=\left\lfloor k+1+\frac{1}{p-1}\right\rfloor+1$, we obtain $n(e,p)\leq \left\lfloor k+1+\frac{1}{p-1}\right\rfloor+1\leq k+2+\frac{1}{p-1}$, as required.
\end{proof}

\section{Optimality}
By constructing pairs of non-isomorphic fields with sufficiently long isomorphic truncations, we can show that the bounds obtained in the previous section are optimal. 

\begin{prop}\label{p=2}
    If $2\mid e$,\[
    n(e,2) = k+3
    \] 
\end{prop}
\begin{prop}\label{oddp}
    For $p\neq2$ and $p\mid e$,\[
    n(e,p) = k+2
    \]
\end{prop}
To prove these two propositions, the following lemma will help:

\begin{lemma}\label{pathology}
Let $e=np^k$, where $(n,p)=1$ and $p\mid e$. For $p$ odd,
\[
    \mathbb{Q}_p(p^{1/e})
    \ncong
    \mathbb{Q}_p((p(1+p)^{e/p})^{1/e}),
\]
and for $p=2$,
\[
    \mathbb{Q}_2(2^{1/e})
    \ncong
    \mathbb{Q}_2((2\cdot5^{e/2})^{1/e}).
\]
\end{lemma}

\begin{proof}[Proof of the odd $p$ case]
Let $\alpha:=p^{1/e}$, $\beta:=(p(1+p)^{e/p})^{1/e}$, $L:=\mathbb{Q}_p(\alpha)$ and $L':=\mathbb{Q}_p(\beta)$. Suppose $L\cong L'$. Then there is an $x\in L'$ satisfying $x^e=p$. Set $y=x/\beta$ and $t=y^n$. Since $e=np^k$, we have $(t^p(1+p)^n)^{p^{k-1}}=1$.

We first show that $L'$ contains no nontrivial root of unity of $p$-power order. Suppose otherwise. Then $\zeta_p\in L'$. Let $\delta:=\beta^{p^k}/(1+p)^{p^{k-1}}$. Then $\delta^n=p$, so $T:=\mathbb{Q}_p(\delta)$ is the maximal tamely ramified subextension of $L'/\mathbb{Q}_p$. Since $\mathbb{Q}_p(\zeta_p)/\mathbb{Q}_p$ is tamely ramified of degree $p-1$, we have $\mathbb{Q}_p(\zeta_p)\subseteq T$. Hence $p-1\mid n$.

Write $n=h(p-1)$ for some positive integer $h$, and set $a:=\delta^h$. Then $a^{p-1}=p$. Now $E:=\mathbb{Q}_p(\zeta_p)=\mathbb{Q}_p(\sqrt[p-1]{-p})$ by \cite[Lemma 14.6]{Washington}, and $E\subseteq T$. Let $b:=\sqrt[p-1]{-p}$ and $z:=b/a$. Then $z^{p-1}=-1$, so $z$ is a root of unity whose order divides $2(p-1)$ but does not divide $p-1$. Since $p$ is odd, the order of $z$ is prime to $p$.

Since $T/\mathbb{Q}_p$ is totally ramified, the residue field of $T$ is $\mathbb{F}_p$. Reduction modulo the maximal ideal gives a group homomorphism $\mathcal O_T^\times\to\mathbb{F}_p^\times$. We claim that this map is injective on roots of unity of order prime to $p$. Indeed, let $\zeta$ be a root of unity of order $m$, where $(m,p)=1$, and suppose $\zeta\equiv1\pmod{\mathfrak m_T}$. Then $\zeta$ is a root of $f(X)=X^m-1$, and modulo $\mathfrak m_T$ the element $1$ is a simple root of $f$, since $f'(1)=m$ is a unit. By Hensel's lemma, $1$ has a unique lift to a root of $f$ in $\mathcal O_T$. Since $1$ itself is such a lift, we must have $\zeta=1$.

Thus $\langle z\rangle$ embeds into $\mathbb{F}_p^\times$. By Lagrange's theorem, the order of $z$ must divide $|\mathbb{F}_p^\times|=p-1$, contradicting the fact that its order does not divide $p-1$. Therefore $\zeta_p\notin L'$, and hence $L'$ contains no nontrivial root of unity of $p$-power order.

It follows from $(t^p(1+p)^n)^{p^{k-1}}=1$ that $t^p=(1+p)^{-n}$. Choose $a,b\in\mathbb{Z}$ such that $ap+bn=1$. Then $r:=t^{-b}(1+p)^a$ satisfies $r^p=1+p$, and hence $F:=\mathbb{Q}_p((1+p)^{1/p})\subseteq L'$.

Writing $r=1+s$, the element $s$ is a root of the Eisenstein polynomial $f(X)=(1+X)^p-(1+p)$. Hence $[F:\mathbb{Q}_p]=p$ and $s$ is a uniformiser of $F$. Since $F=\mathbb{Q}_p(s)$, the different is generated by $f'(s)$ by \cite[Chapter III, \S6, Corollary 2 to Proposition 11]{serre1979local}. Now $f'(X)=p(1+X)^{p-1}$, so $f'(s)=pr^{p-1}$. Since $r$ is a unit in $\mathcal O_F$, we obtain $v_F(f'(s))=v_F(p)=p$, as $F/\mathbb{Q}_p$ is totally ramified of degree $p$. Therefore $d(F/\mathbb{Q}_p)=p$.

The minimal polynomial of $\beta$ is $X^e-p(1+p)^{e/p}$, so $d(L'/\mathbb{Q}_p)=v_{L'}(e\beta^{e-1})=ek+e-1$ by \cite[Chapter III, \S6, Corollary 2 to Proposition 11]{serre1979local}. By transitivity of the different \cite[Chapter III, \S4, Proposition 8]{serre1979local}, we therefore obtain $d(L'/F)=ek-1$.

On the other hand, $[L':F]=e/p$, so the standard bound for the different \cite[Chapter III, \S6, Remark 1 following Proposition 13]{serre1979local} gives $d(L'/F)\leq e/p-1+e(k-1)<ek-1$, a contradiction. Therefore $L\ncong L'$.
\end{proof}

\begin{proof}[Proof of the $p=2$ case]
Let $L=\mathbb{Q}_2(2^{1/e})$ and $L'=\mathbb{Q}_2((2\cdot5^{e/2})^{1/e})$, and let $\alpha,\beta$ be elements in $L,L'$ respectively such that $\alpha^e=2$ and $\beta^e=2\cdot5^{e/2}$. Suppose that $L\cong L'$. Then there is some $x\in L'$ satisfying $x^e=2$. Set $y=x/\beta$ and $t=y^n$. Then $(t^25^n)^{2^{k-1}}=1$.

Let $\xi:=t^25^n$. If $\xi$ has order at least $4$, then, since its order is a power of $2$, some power of $\xi$ is a primitive fourth root of unity. Hence $i\in L'$. Since $d(\mathbb{Q}_2(i)/\mathbb{Q}_2)=2$, transitivity of the different gives $d(L'/\mathbb{Q}_2(i))=ek-1$. The standard bound for the different \cite[Chapter III, \S6, Remark 1 following Proposition 13]{serre1979local} then gives $d(L'/\mathbb{Q}_2(i))\leq e/2-1+e(k-1)<ek-1$, a contradiction.

Hence $\xi=\pm1$. Suppose first that $\xi=1$. Since $n$ is odd, $(n+1)/2\in\mathbb{Z}$, and $(t5^{(n+1)/2})^2=5$. Thus $\sqrt{5}\in L'$. However, $\mathbb{Q}_2(\sqrt{5})/\mathbb{Q}_2$ is unramified, whereas $L'/\mathbb{Q}_2$ is totally ramified, a contradiction.

Suppose instead that $\xi=-1$. Then $(t5^{(n+1)/2})^2=-5$, so $\mathbb{Q}_2(\sqrt{-5})\subseteq L'$. Writing $s:=\sqrt{-5}-1$, the element $s$ is a root of the Eisenstein polynomial $f(X)=X^2+2X+6$, and hence is a uniformiser of $\mathbb{Q}_2(\sqrt{-5})$. By \cite[Chapter III, \S6, Corollary 2 to Proposition 11]{serre1979local}, its different is generated by $f'(s)=2(s+1)=2\sqrt{-5}$. Since $\sqrt{-5}$ is a unit, we obtain $d(\mathbb{Q}_2(\sqrt{-5})/\mathbb{Q}_2)=2$. Hence, by transitivity of the different, $d(L'/\mathbb{Q}_2(\sqrt{-5}))=ek-1$. On the other hand, the different bound gives $d(L'/\mathbb{Q}_2(\sqrt{-5}))\leq e/2-1+e(k-1)<ek-1$, again a contradiction.

Thus every possibility leads to a contradiction, and therefore $L\ncong L'$.
\end{proof}

\begin{proof}[Proof of Proposition \ref{p=2}]
As a consequence of Theorem \ref{mainbound}, when $2\mid e$ we have $n(e,2)\leq k+3$. It therefore suffices to show that $n(e,2)>k+2$.

Let $L$ and $L'$ be the two fields from Lemma \ref{pathology}. Since both extensions are defined by Eisenstein polynomials, \cite[Chapter I, \S6, Proposition 17 and the following Corollary]{serre1979local} gives $\mathcal O_L\cong\mathbb{Z}_2[X]/(X^e-2)$ and $\mathcal O_{L'}\cong\mathbb{Z}_2[X]/(X^e-2\cdot5^{e/2})$.

Since $e/2=n2^{k-1}$ and $n$ is odd, we have $v_2(e/2)=k-1$. For every positive integer $m$, the lifting-the-exponent lemma \cite[\S1.9, Theorem 1.37]{pongsriiam} gives $v_2(5^m-1)=v_2(5-1)+v_2(m)=2+v_2(m)$. Taking $m=e/2$, we obtain $v_2(5^{e/2}-1)=k+1$, and hence $v_2(2\cdot5^{e/2}-2)=k+2$. Therefore $2\cdot5^{e/2}\equiv2\pmod{2^{k+2}}$.

It follows that $X^e-2\equiv X^e-2\cdot5^{e/2}\pmod{2^{k+2}}$, and hence
\[
\mathcal O_L/(2^{k+2})
\cong
\frac{\mathbb Z_2[X]}{(X^e-2,\,2^{k+2})}
\cong
\frac{\mathbb Z_2[X]}{(X^e-2\cdot5^{e/2},\,2^{k+2})}
\cong
\mathcal O_{L'}/(2^{k+2}).
\]
By Lemma \ref{pathology}, $L\ncong L'$. Thus $n(e,2)>k+2$, and together with the upper bound $n(e,2)\leq k+3$, this gives $n(e,2)=k+3$.
\end{proof}

\begin{proof}[Proof of Proposition \ref{oddp}]
By Theorem \ref{mainbound}, for odd $p$ we have $n(e,p)\leq k+2$. It therefore suffices to show that $n(e,p)>k+1$.

Let $L$ and $L'$ be the two fields from Lemma \ref{pathology}. As above, $\mathcal O_L\cong\mathbb{Z}_p[X]/(X^e-p)$ and $\mathcal O_{L'}\cong\mathbb{Z}_p[X]/(X^e-p(1+p)^{e/p})$.

By the lifting-the-exponent lemma, $v_p((1+p)^{e/p}-1)=1+v_p(e/p)=k$. Hence $v_p(p(1+p)^{e/p}-p)=k+1$, so $p(1+p)^{e/p}\equiv p\pmod{p^{k+1}}$.

Therefore
\[
\mathcal O_L/(p^{k+1})
\cong
\frac{\mathbb Z_p[X]}{(X^e-p,\,p^{k+1})}
\cong
\frac{\mathbb Z_p[X]}{(X^e-p(1+p)^{e/p},\,p^{k+1})}
\cong
\mathcal O_{L'}/(p^{k+1}).
\]
Since $L\ncong L'$ by Lemma \ref{pathology}, we have $n(e,p)>k+1$. Combining this with $n(e,p)\leq k+2$ gives $n(e,p)=k+2$.
\end{proof}

\section*{Acknowledgments}
The author would like to thank Philip Dittmann for suggesting this problem and for regular helpful discussions and feedback. The author has discussed and exchanged ideas with an LLM and in parts the LLM has helped find references for shorter proofs in the literature, errors in proofs as well as finding the final two examples for optimality. The author takes full responsibility for all parts of the work.
\printbibliography
\end{document}